\documentclass[10pt]{amsart}

\usepackage{mathpazo}
\usepackage{amsmath,amssymb,amscd,color}
\usepackage{mathrsfs}
\usepackage{float}
\usepackage{tikz}
\usetikzlibrary{positioning}
\usetikzlibrary{decorations.markings}
\usetikzlibrary{arrows, scopes}
\usepackage{comment}
\definecolor{green}{RGB}{0,144,0}
\definecolor{bluegreen}{RGB}{17,100,180}
\usepackage[linkcolor = bluegreen, citecolor = bluegreen, colorlinks = True, backref = page]{hyperref}
\numberwithin{equation}{section}
\numberwithin{figure}{section}
\numberwithin{table}{section}
\usepackage[all]{hypcap}
\usepackage{mathtools}

\usepackage{xcolor}
\definecolor{codegreen}{rgb}{0,0.6,0}
\definecolor{codegray}{rgb}{0.5,0.5,0.5}
\definecolor{codepurple}{rgb}{0.58,0,0.82}
\definecolor{backcolour}{rgb}{0.95,0.95,0.92}
\definecolor{flame}{rgb}{0.89, 0.35, 0.13}

\newtheorem{theorem}{Theorem}[section]
\newtheorem*{theorem*}{Theorem}

\newtheorem{lemma}[theorem]{Lemma}

\newtheorem{proposition}[theorem]{Proposition}
\newtheorem{claim}[theorem]{Claim}
\newtheorem{question}[theorem]{Question}

\newtheorem{remark}[theorem]{Remark}

\newcommand{\teichmuller}{Teichm{\"u}ller{ }}

\newcommand{\calH}{\mathcal{H}}
\newcommand{\calQ}{\mathcal{Q}}

\newcommand{\hyp}{\mathcal{H}^{hyp}}
\newcommand{\Qhyp}{\mathcal{Q}^{hyp}}

\newcommand{\T}{\mathcal{T}}
\newcommand{\C}{\mathcal{C}}
\newcommand{\D}{\mathcal{D}}

\newcommand{\Cbb}{\mathbb{C}}
\newcommand{\R}{\mathbb{R}}
\newcommand{\Hbb}{\mathbb{H}}
\newcommand{\sys}{\text{sys}}
\newcommand{\Mod}{\text{Mod}}
\newcommand{\SL}{\text{SL}}
\newcommand{\SO}{\text{SO}}

\newcommand{\bs}{\boldsymbol}

\title{Meanders, hyperelliptic pillowcase covers, and the Johnson filtration}

\author{Luke Jeffreys}
\address{School of Mathematics , University of Bristol, Fry Building, Woodland Road, Bristol BS8 1UG}
\curraddr{}
\email{luke.jeffreys@bristol.ac.uk}
\thanks{The author is thankful for support from the Heilbronn Institute for Mathematical Research while this work was undertaken. The author is currently a Leverhulme Early Career Fellow (ECF-2023-553) and so also wishes to thank the Leverhulme Trust.}

\subjclass[2020]{Primary: 32G15, 30F30, 30F60. Secondary: 57M50.}

\begin{document}

\begin{abstract}
We provide minimal constructions of meanders with particular combinatorics. Using these meanders, we give minimal constructions of hyperelliptic pillowcase covers with a single horizontal cylinder and simultaneously a single vertical cylinder so that one or both of the core curves are separating curves on the underlying surface. In the case where both of the core curves are separating, we use these surfaces in a construction of Aougab-Taylor in order to prove that for any hyperelliptic connected component of the moduli space of quadratic differentials with no poles there exist ratio-optimising pseudo-Anosovs lying arbitrarily deep in the Johnson filtration and stabilising the Teichm{\"u}ller disk of a quadratic differential lying in this connected component.
\end{abstract}

\maketitle


\section{Introduction}

Meanders are special planar curve and arc systems that have been historically well-studied in many areas of mathematics, computer science and physics. Indeed, it is possible to argue that they were studied as early as 1912 by Poincaré in his investigations of certain annular diffeomorphisms~\cite{P}. More recently, they have been studied by Arnol'd (who introduced the term meander) in the setting of algebraic geometry~\cite{Arn}, and by Ko-Smolinsky in the study of 3-manifold invariants~\cite{Ko-S}. In computer science, they are related to the study of Jordan sequences, while in physics they are relevant to the study of compact polymer folding, the Temperley-Lieb algebra, matrix models, and (2+1)-dimensional gravity. Moreover, their enumeration still proves to be a challenging open problem. We direct the reader to the works of Lando-Zvonkin~\cite{LZ}, Di Francesco-Golinelli-Guitter~\cite{DiFGG1,DiFGG2}, and Jensen~\cite{Jen} for more details.

For the purposes of the current work, we are interested in the fact that meanders can be viewed as filling pairs on the punctured sphere. Specifically, a pair of non-homotopically trivial simple closed curves on a surface, minimising their intersections up to homotopy, are said to be a filling pair if their complement is a disjoint union of at most once-punctured disks. One can extend this definition to the notion of a filling set of curves (a pair being the case where the set is of size two). Each intersection between the curves in a filling set is 4-valent and so, since the complementary regions are at most-once punctured disks, the dual complex on the surface is a tiling of the surface by squares. Realising each square as a unit Euclidean square, we realise the surface itself as a pillowcase cover - that is, as a branched cover of the four-times punctured sphere (the pillowcase). Equivalently, we equip the surface with a meromorphic quadratic differential that is a cover of a quadratic differential on the pillowcase. So meanders, being filling pairs on the sphere, give rise to pillowcase covers of genus zero. In fact, coming specifically from filling pairs, these are examples of a type of pillowcase cover called a $[1,1]$-pillowcase cover. A $[1,1]$-pillowcase cover is a pillowcase cover that simultaneously has a single horizontal cylinder (coming from one of the curves in the filling pair) and a single vertical cylinder (coming from the other curve in the filling pair). See Section~\ref{s:QDs} for more details and definitions.

The moduli space of integrable meromorphic quadratic differentials on a closed Riemann surface of genus $g\geq 0$ is a disjoint union
\[\calH_{g}\sqcup\calQ_{g},\]
where $\calH_{g}$ is the subset of those quadratic differentials that are global squares of Abelian differentials. Note that $\calH_{g}$ is empty in genus zero. Again, we direct the reader to Section~\ref{s:QDs} for more details - we include only enough here to be able to state our main results. Both components in this union are themselves disjoint unions of subsets called strata, where the disjoint union is parameterised by the orders of the zeros of the quadratic differential:
\[\calH_{g}=\bigsqcup_{\substack{k_{1}+\cdots+k_{n} = 2g-2\\k_{i}\geq 1}}\calH(k_{1},\ldots,k_{n}),\]
and
\[Q_{g}=\bigsqcup_{\substack{k_{1}+\cdots+k_{n} = 4g-4\\k_{i}\in\{-1\}\cup\{1,2,\ldots\}}}\calQ(k_{1},\ldots,k_{n}).\]
Moreover, the strata themselves can be disconnected. The classification of the connected components of strata was carried out by Kontsevich-Zorich~\cite{KZ} in the case of $\calH_{g}$, and by Lanneau~\cite{Lan1,Lan2} and Chen-M{\"o}ller~\cite{CM} in the case of $Q_{g}$.

Given a connected component of a stratum of the moduli space, the pillowcase covers can be thought of as the rational points in a particular coordinate system. As such, pillowcase covers play a key role in the study of certain properties of these components. For example, the pillowcase covers in the connected components of $\calH_{g}$, often called square-tiled surfaces in this context, were counted in the works of Eskin-Okounkov~\cite{EO} and Zorich~\cite{Z1} in order to determine the volumes of the components (with respect to the Masur-Veech measure). Using the correspondence between meanders and genus zero pillowcase covers discussed above, the enumeration of meanders was utilised by Delecroix-Goujard-Zograf-Zorich~\cite{DGZZ} to calculate the volumes of strata of the moduli space of genus zero quadratic differentials.

In many cases, restricting the combinatorics of the pillowcase covers that one is interested in can be useful. An important piece of combinatorial data is the number of maximally embedded flat annuli (called cylinders) in the horizontal and vertical directions. Arguably then, the combinatorially simplest are those said to be $[1,1]$-pillowcase covers. These are those pillowcase covers that have simultaneously a single horizontal cylinder and a single vertical cylinder. Equivalently, these are the pillowcase covers obtained from filling pairs via the dual square-tiling discussed above. A natural question to ask is, given a connected component of the moduli space of quadratic differentials, what is the minimum number of squares required to construct a $[1,1]$-pillowcase cover in that connected component. When both of the core curves of the cylinders are non-separating curves, this question was answered by the author in $\calH_{g}$~\cite{J1} -- where having both curves be non-separating is forced upon us -- and also in certain components of $\calQ_{g}$ in his thesis~\cite{J0} and a separate work~\cite{J2}. As the non-separating condition on the core curves is not forced upon us in $\calQ_{g}$, it is natural to ask the following question:

\begin{question}[A variant of {\cite[Question 7.2]{J0}}]
What is the minimum number of squares required to construct a $[1,1]$-pillowcase cover in a connected component of $\calQ_{g}$ where one or both of the core curves are separating?
\end{question}

In this paper, we will answer this question for the hyperelliptic connected components.

The notion of hyperellipticity is used in the classifications of Kontsevich-Zorich, and Lanneau and Chen-M{\"o}ller. Indeed, it was shown that the strata $\calH(2g-2),\calH(g-1,g-1),\calQ(4j+2,4k+2),\calQ(4j+2,2k-1,2k-1)$, and $\calQ(2j-1,2j-1,2k-1,2k-1)$ all have connected components consisting entirely of hyperelliptic quadratic differentials -- differentials that are specific double covers of quadratic differentials of genus zero. These connected components are called the hyperelliptic components and are denoted by $\hyp(2g-2),\hyp(g-1,g-1),\Qhyp(4j+2,4k+2),\Qhyp(4j+2,2k-1,2k-1)$, and $\Qhyp(2j-1,2j-1,2k-1,2k-1)$, respectively.

The main result of this paper is then the following.

\begin{theorem}\label{th:main}
The minimum number of squares required to produce $[1,1]$-pillowcase covers in the hyperelliptic components of the strata of the moduli space of quadratic differentials are as in Table~\ref{tab:theorem}.
\end{theorem}
\begin{table}[H]
\centering
\addtolength{\leftskip} {-4cm}
\addtolength{\rightskip}{-4cm}
\begin{tabular*}{1.4\textwidth}{@{\extracolsep{\fill}}|c|c|c|c|}
\hline
 & \multicolumn{3}{c|}{Minimum number of squares required to produce a} \\
 Connected component & \multicolumn{3}{c|}{$[1,1]$-pillowcase cover whose cylinders are} \\
 & both non-sep. & one sep. one non-sep. & both sep. \\
\hline
 $\hyp(2g-2), g\geq 2$ & $4g-4$ & n/a & n/a \\
 $\hyp(g-1,g-1), g\geq 2$ & $4g-2$ & n/a & n/a \\
 $\Qhyp(4j+2,4k+2),k\geq j\geq 0$ & $4j+4k+4$ & $\max\{8j+6,8k+4\}$ & $16k-8j+8$ \\
 $\Qhyp(4j+2,2k-1,2k-1),j\geq 1,k\geq 0,j\geq k$ & $4j+4k+2$ & $8j+4$ & $16j-8k+12$ \\
 $\Qhyp(4j+2,2k-1,2k-1),k>j\geq 0$ & $4j+4k+2$ & $8k$ & $16k-8j$ \\
 $\Qhyp(2j-1,2j-1,2k-1,2k-1),k\geq 1, j\geq 0, k\geq j$ & $4j+4k$ & $\max\{8j+2,8k\}$ & $16k-8j+4$ \\
 $\Qhyp(2,-1,-1)$ & $3$ & $4$ & $12$ \\
\hline
\end{tabular*}
\vspace*{10pt}
\caption{Minimum number of squares required for hyperelliptic $[1,1]$-pillowcase covers depending on the separating properties of their cylinders}\label{tab:theorem}
\end{table}

We find that, as you increase the number of separating core curves from zero to two, the number of squares required increasingly depends on how the zeros are related to one another (from no relation, to maximum, to difference).

\subsection*{Meanders in the proof of Theorem~\ref{th:main}}

As alluded to above, the core curves of the cylinders of a $[1,1]$-pillowcase cover form a filling pair on the surface. For a hyperelliptic $[1,1]$-pillowcase cover, these curves descend under the double cover associated to a hyperelliptic involution to a curve and arc system on a puncture sphere. In fact, they descend to an arc and curve system that corresponds to a (potentially open or semi-) meander with very specific combinatorics. We consider minimal constructions of meanders with the necessary combinatorics. By taking double covers of these meanders we obtain $[1,1]$-pillowcase covers in the hyperelliptic components of the moduli space of quadratic differentials. Moreover, these constructions are minimal since they arose from minimal constructions of the associated meanders.

\subsection*{Ratio-optimising pseudo-Anosovs in the Johnson filtration}

Ratio-optimising \linebreak pseudo-Anosovs are pseudo-Anosov homeomorphisms of a surface that optimise a geometrically significant ratio related to the action of the homeomorphism on Teichm{\"u}ller space and the curve graph. We direct the reader to Section~\ref{s:ROs} for the explicit details.

Such pseudo-Anosovs were utilised by Gadre-Hironaka-Kent-Leininger~\cite{GHKL} in their determination of the optimum Lipschitz constant for the systole map from Teichm{\"u}ller space to the curve graph. Following this work, Aougab-Taylor~\cite{AT} gave a construction of ratio-optimising pseudo-Anosovs using a Thurston construction on suitable filling pairs on a surface. Moreover, they showed that when both of the curves in the filling pair are separating curves then ratio-optimising pseudo-Anosovs can be produced from this filling pair all of which will stabilise the Teichm{\"u}ller disk of the pillowcase cover associated to the filling pair and can be taken to lie arbitrarily deep in the Johnson filtration of the mapping class group (see Subsection~\ref{ss:Jfilt} for definitions). A natural question to ask is the following:

\begin{question}[{\cite[Question 7.4]{J0}}]
Which connected components of the moduli space of quadratic differentials contain quadratics differentials whose \teichmuller disks are stabilised by ratio-optimising pseudo-Anosovs lying arbitrarily deep in the Johnson filtration?
\end{question}

Since $[1,1]$-pillowcase covers correspond to filling pairs on surfaces, we can use the filling pairs corresponding to $[1,1]$-pillowcase covers where both core curves are separating in this construction of Aougab-Taylor in order to prove the following result.

\begin{theorem}\label{th:RO}
Let $\mathscr{C}$ be any hyperelliptic connected component of the moduli space of quadratic differentials $\calQ_{g}$ with no poles. There exist infinitely many ratio-optimising pseudo-Anosov homeomorphisms lying arbitrarily deep in the Johnson filtration of the mapping class group and stabilising the \teichmuller disk of a quadratic differential from $\mathscr{C}$.
\end{theorem}

That is, we answer the above question for the hyperelliptic connected components. The case of non-hyperelliptic connected components remains open.

\subsection*{Outline of the paper}

In Section~\ref{s:QDs}, we give the necessary background on quadratic differentials and pillowcase covers. In Section~\ref{s:meanders}, we give the definitions of the meanders and required variations that are of interest to this work. In Section~\ref{s:constructions}, we will give the minimal constructions of the relevant meanders and describe how to lift these to the $[1,1]$-pillowcase covers in the hyperelliptic components - that is, we prove Theorem~\ref{th:main}. In Section~\ref{s:ROs}, we describe the applications of these constructions to the study of ratio-optimising pseudo-Anosovs in the Johnson filtration and prove Theorem~\ref{th:RO}. Finally, in Section~\ref{s:othermeanders} we consider minimal constructions of meanders in more general strata than those listed in Table~\ref{tab:theorem}.

\subsection*{Acknowledgements}

We thank Tara Brendle for useful discussions related to \linebreak Lemma~\ref{l:lift}. We also thank the referee for their careful reading and their suggestions that have helped to improve the clarity of the proofs of Propositions~\ref{p:closedmeander1} and~\ref{p:closedmeander2}.

\section{Quadratic differentials}\label{s:QDs}

Here we give the necessary background on quadratic differentials and pillowcase covers. We direct the reader to the works of Forni-Matheus~\cite{FoM}, Zorich~\cite{Z2}, and Strebel~\cite{S} for more details.

A {\em quadratic differential} $q$ on a compact Riemann surface $X$ of genus $g\geq 0$ is a global section of the symmetric square of the canonical line bundle $\Omega(X)$. That is, in local coordinates $q$ is given by $f(z)dz^{2}$. An {\em Abelian differential} on $X$ is a non-zero holomorphic one-form $\omega$ locally given by $\omega=f(z)dz$ for a holomorphic function $f$. The global square of an Abelian differential gives rise to a quadratic differential on a Riemann surface. We define the moduli space of integrable meromorphic quadratic differentials on a Riemann surface $X$ to be the quotient by the action of the mapping class group of the set of pairs $(X,q)$, where $X$ is a closed connected Riemann surface of genus $g$ and $q$ is a non-zero meromorphic quadratic differential on $X$ having at most simple poles. As discussed in the introduction, this moduli space is then the disjoint union
\[\calH_{g}\sqcup\calQ_{g},\]
where $\calH_{g}$ is the subset of those quadratic differentials that are global squares of Abelian differentials.

\begin{figure}[b]
\begin{center}
\begin{tikzpicture}[scale = 0.6]
\draw [line width = 0.3mm, line cap = round] (0,0)--(1,3);
\draw (0.5-0.3,1.5+0.1) node{0};
\draw (1,3) node{$\bullet$};
\draw [line width = 0.3mm, line cap = round] (1,3)--(2,3.2);
\draw (1.5,3.5) node{1};
\draw (2,3.2) node{$\bullet$};
\draw [line width = 0.3mm, line cap = round] (2,3.2)--(3.4,2);
\draw (2.9,2.9) node{2};
\draw (3.4,2) node{$\bullet$};
\draw [line width = 0.3mm, line cap = round] (3.4,2)--(4.5,3);
\draw (3.95-0.25,2.5+0.275) node{3};
\draw (4.5,3) node{$\bullet$};
\draw [line width = 0.3mm, line cap = round] (4.5,3)--(6,1);
\draw (0.75+0.3+4.5,-1+0.225+3) node{4};
\draw (6,1) node{$\bullet$};
\draw [line width = 0.3mm, line cap = round] (6,1)--(5,-2);
\draw (0.5+0.3+5,1.5-0.1-2) node{0};
\draw (5,-2) node{$\bullet$};
\draw [line width = 0.3mm, line cap = round] (2.6,-1)--(3.6,-0.8);
\draw (3.1,-1.3) node{1};
\draw (2.6,-1) node{$\bullet$};
\draw [line width = 0.3mm, line cap = round] (5,-2)--(3.6,-0.8);
\draw (4.1,-1.7) node{2};
\draw (3.6,-0.8) node{$\bullet$};
\draw [line width = 0.3mm, line cap = round] (2.6,-1)--(1.5,-2);
\draw (3.95+0.25-1.9,2.5-0.275-4) node{3};
\draw (1.5,-2) node{$\bullet$};
\draw [line width = 0.3mm, line cap = round] (1.5,-2)--(0,0);
\draw (0.75-0.3,-1-0.225) node{4};
\draw (0,0) node{$\bullet$};
\draw [line width = 0.3mm, line cap = round] (7+1,0)--(8+1,3);
\draw (7.5-0.3+1,1.5+0.1) node{0};
\draw (9,3) node{$\bullet$};
\draw [line width = 0.3mm, line cap = round, ->] (8+1,3)--(8.6+1,2.75);
\draw [line width = 0.3mm, line cap = round] (8.6+1,2.75)--(9.2+1,2.5);
\draw (8.6+0.15+1,2.75+0.36) node{1};
\draw (10.2,2.5) node{$\bullet$};
\draw [line width = 0.3mm, line cap = round] (9.2+1,2.5)--(9.8+1,2.25);
\draw [line width = 0.3mm, line cap = round, <-] (9.8+1,2.25)--(10.4+1,2);
\draw (9.8+0.15+1,2.25+0.36) node{1};
\draw (11.4,2) node{$\bullet$};
\draw [line width = 0.3mm, line cap = round] (10.4+1,2)--(11.5+1,3);
\draw (10.95-0.25+1,2.5+0.275) node{2};
\draw (12.5,3) node{$\bullet$};
\draw [line width = 0.3mm, line cap = round] (11.5+1,3)--(13+1,1);
\draw (7.75+0.3+4.5+1,-1+0.225+3) node{3};
\draw (14,1) node{$\bullet$};
\draw [line width = 0.3mm, line cap = round] (13+1,1)--(12+1,-2);
\draw (7.5+0.3+5+1,1.5-0.1-2) node{0};
\draw (13,-2) node{$\bullet$};
\draw [line width = 0.3mm, line cap = round] (10.8+1,-1.5)--(8.6+1.6+1,2.75-4);
\draw [line width = 0.3mm, line cap = round, ->] (9.6+1,-1)--(8.6+1.6+1,2.75-4);
\draw (8.6-0.15+1.6+1,2.75-0.36-4) node{4};
\draw (10.6,-1) node{$\bullet$};
\draw [line width = 0.3mm, line cap = round, ->] (12+1,-2)--(9.8+1.6+1,2.25-4);
\draw [line width = 0.3mm, line cap = round] (9.8+1.6+1,2.25-4)--(10.8+1,-1.5);
\draw (9.8-0.15+1.6+1,2.25-0.36-4) node{4};
\draw (11.8,-1.5) node{$\bullet$};
\draw [line width = 0.3mm, line cap = round] (9.6+1,-1)--(8.5+1,-2);
\draw (10.95+0.25-1.9+1,2.5-0.275-4) node{2};
\draw (9.5,-2) node{$\bullet$};
\draw [line width = 0.3mm, line cap = round] (8.5+1,-2)--(7+1,0);
\draw (7.75-0.3+1,-1-0.225) node{3};
\draw (8,0) node{$\bullet$};
\end{tikzpicture}
\end{center}
\caption{Two half-translation surfaces. Sides with the same label are identified by translation, or by half-translation when arrows are indicated. The surface on the left is also a translation surface.}
\label{f:half-translationsurfaces}
\end{figure}

\begin{figure}[t]
\begin{center}
\begin{tikzpicture}[scale = 1.25]
\draw [line width = 0.3mm, line cap = round] (0,0)--node[left]{0}(0,1);
\draw [line width = 0.3mm, line cap = round] (0,1)--node[above]{1}(1,1);
\draw [line width = 0.3mm, line cap = round] (1,1)--node[above]{2}(2,1);
\draw [line width = 0.3mm, line cap = round] (2,1)--node[above]{1}(3,1);
\draw [line width = 0.3mm, line cap = round] (3,1)--node[above]{2}(4,1);
\draw [line width = 0.3mm, line cap = round] (4,1)--node[above]{3}(5,1);
\draw [line width = 0.3mm, line cap = round] (5,1)--node[right]{0}(5,0);
\draw [line width = 0.3mm, line cap = round] (5,0)--node[below]{3}(4,0);
\draw [line width = 0.3mm, line cap = round] (4,0)--node[below]{5}(3,0);
\draw [line width = 0.3mm, line cap = round] (3,0)--node[below]{4}(2,0);
\draw [line width = 0.3mm, line cap = round] (2,0)--node[below]{5}(1,0);
\draw [line width = 0.3mm, line cap = round] (1,0)--node[below]{4}(0,0);
\draw [densely dashed](1,0)--(1,1);
\draw [densely dashed](2,0)--(2,1);
\draw [densely dashed](3,0)--(3,1);
\draw [densely dashed](4,0)--(4,1);
\draw [line width = 0.3mm, color = red] (0,0.5)--(5,0.5);
\draw [line width = 0.3mm, color = blue] (0.5,0)--(0.5,1);
\draw [line width = 0.3mm, color = green] (1.5,0)--(1.5,1);
\draw [line width = 0.3mm, color = blue] (2.5,0)--(2.5,1);
\draw [line width = 0.3mm, color = green] (3.5,0)--(3.5,1);
\draw [line width = 0.3mm, color = flame] (4.5,0)--(4.5,1);
\foreach \x in {0,1,2,3,4,5}
	\draw (\x,0) node{$\bullet$};
\foreach \x in {0,1,2,3,4,5}
	\draw (\x,1) node{$\bullet$};
\end{tikzpicture}
\begin{tikzpicture}[scale = 1.25]
\draw [line width = 0.3mm, line cap = round] (0,0)--node[left]{0}(0,1);
\draw [line width = 0.3mm, line cap = round] (0,1)--node[above]{1}(1,1);
\draw [line width = 0.3mm, line cap = round] (1,1)--node[above]{2}(2,1);
\draw [line width = 0.3mm, line cap = round] (2,1)--node[above]{3}(3,1);
\draw [line width = 0.3mm, line cap = round] (3,1)--node[above]{4}(4,1);
\draw [line width = 0.3mm, line cap = round] (4,1)--node[above]{3}(5,1);
\draw [line width = 0.3mm, line cap = round] (5,1)--node[right]{0}(5,0);
\draw [line width = 0.3mm, line cap = round] (5,0)--node[below]{5}(4,0);
\draw [line width = 0.3mm, line cap = round] (4,0)--node[below]{1}(3,0);
\draw [line width = 0.3mm, line cap = round] (3,0)--node[below]{4}(2,0);
\draw [line width = 0.3mm, line cap = round] (2,0)--node[below]{5}(1,0);
\draw [line width = 0.3mm, line cap = round] (1,0)--node[below]{2}(0,0);
\draw [densely dashed](1,0)--(1,1);
\draw [densely dashed](2,0)--(2,1);
\draw [densely dashed](3,0)--(3,1);
\draw [densely dashed](4,0)--(4,1);
\draw [line width = 0.3mm, color = red] (0,0.5)--(5,0.5);
\draw [line width = 0.3mm, color = green] (0.5,0)--(0.5,1);
\draw [line width = 0.3mm, color = green] (1.5,0)--(1.5,1);
\draw [line width = 0.3mm, color = green] (2.5,0)--(2.5,1);
\draw [line width = 0.3mm, color = green] (3.5,0)--(3.5,1);
\draw [line width = 0.3mm, color = green] (4.5,0)--(4.5,1);
\foreach \x in {0,1,2,3,4,5}
	\draw (\x,0) node{$\bullet$};
\foreach \x in {0,1,2,3,4,5}
	\draw (\x,1) node{$\bullet$};
\end{tikzpicture}
\end{center}
\caption{Two pillowcase covers. The surface at the top has a single-horizontal cylinder but two vertical cylinders. The surface at the bottom is a $[1,1]$-pillowcase cover.}
\label{f:pillowcasecovers}
\end{figure}

A Riemann surface equipped with a quadratic differential can be realised as the quotient of a collection of Euclidean polygons in $\Cbb$ with pairs of parallel sides of equal length identified by half-translations ($z\mapsto \pm z + c$) such that the quotient is a closed connected oriented surface. For this reason, such a Riemann surface is also called {\em half-translation surfaces}. See Figure~\ref{f:half-translationsurfaces} for some examples.

We remark that if all of the side identifications are simply translations ($z\mapsto z+c$), then the obtained quadratic differential is the global square of an Abelian differential and therefore lies in $\calH_{g}$. In this case, the surface is called a {\em translation surface}. See the surface in the left of Figure~\ref{f:half-translationsurfaces}.

\subsection{Pillowcase covers}

Recall from the introduction that a {\em pillowcase cover} is a half-translation surface realised as the quotient of a collection of unit squares. See Figure~\ref{f:pillowcasecovers} for some examples. When drawing pillowcase covers, we do not specify arrows on sides identified by half-translation if it is clear from the context. In the case of translation surfaces, pillowcase covers are more commonly called {\em square-tiled surfaces}.

A cylinder in a half-translation surface is a maximally embedded flat annulus. If a pillowcase cover has a single vertical cylinder and a single horizontal cylinder then we shall call it a {\em $[1,1]$-pillowcase cover}. See the surface at the bottom of Figure~\ref{f:pillowcasecovers}. The core curves of the cylinders of a $[1,1]$-pillowcase cover form a filling pair on the underlying surface. The core curves of $[1,1]$-square-tiled surfaces are forced to be non-separating. This is because the sides of the squares can only be identified by translation and so all sides on one side of a core curve must also occur on the other side. Equivalently, all of the intersections of the associated filling pair occur with the same orientation. In the case of $[1,1]$-pillowcase covers that are not square-tiled surfaces, the core curves need not be separating. Indeed, one, both, or neither of the curves may be separating.

The connection between the stratum of a $[1,1]$-pillowcase cover and the combinatorics of the core curves is very explicit. Indeed, each complementary region of the filling pair obtained from the core curves contains a single vertex of the polygon (up to identification). Each vertex corresponds to a zero of order $k$ for the quadratic differential if and only if the boundary of its complementary region is a $(2k+4)$-gon. So, poles correspond to bigons, regular points correspond to squares, zeros of order 1 correspond to hexagons, and so on.

\subsection{Hyperelliptic connected components}\label{s:hypcomp}

By the Riemann-Roch theorem, we have that the sum of the orders of the zeros of a quadratic differential on a Riemann surface of genus $g$ is equal to $4g-4$. We define the stratum $\calQ(k_{1},\ldots,k_{n})\subset\calQ_{g}$, with $k_{i}\geq 1$ or $k_{i}=-1$ and $\sum_{i=1}^{n}k_{i} = 4g-4$, to be the subset of $\calQ_{g}$ consisting of quadratic differentials with $n$ distinct zeros of orders $k_{1},\ldots,k_{n}$. For $\calH_{g}$, since these quadratic differentials are squares of Abelian differentials and the sum of the orders of the zeros of an Abelian differential is $2g-2$, we define the stratum $\calH(k_{1},\ldots,k_{n})\subset\calH_{g}$, with $k_{i}\geq 1$ and $\sum_{i=1}^{n}k_{i} = 2g-2$, to be the subset of $\calH_{g}$ consisting of quadratic differentials that are squares of Abelian differentials with $n$ distinct zeros of orders $k_{1},\ldots,k_{n}$. If an order is repeated a number of times we may use exponentiation notation; for example, we might write $\calQ(1,-1,-1,-1,-1,-1)$ as $\calQ(1,-1^{5})$.

The classification of the connected components of strata of the moduli space of quadratic differentials was completed by Kontsevich-Zorich~\cite{KZ} in the case of $\calH_{g}$ and by Lanneau~\cite{Lan1,Lan2}, with a correction by Chen-M{\"o}ller~\cite{CM}, in the case of $\calQ_{g}$. In the case of $\calH_{g}$, connected components were determined by using the notions of spin parity and hyperellipticity. In $\calQ_{g}$, outside of a small number of exceptional strata in low genus, hyperellipticity is sufficient to determine the connected components.

We say that a half-translation surface $(X,q)$ is {\em hyperelliptic} if there exists an isometric involution $\tau:X\to X$, known as a hyperelliptic involution, that induces a ramified double cover $\pi:X\to S_{0,2g+2}$ from $X$ to the $(2g+2)$-times punctured sphere. There exists a double covering construction that takes a quadratic differential $(X_{0},q_{0})$ on the sphere and gives a quadratic differential $(X,q)$ on a higher genus Riemann surface. In the cases of
\[\calQ(2g-3,-1^{2g+1})\to \calH(2g-2)\]
and
\[\calQ(2g-2,-1^{2g+2})\to\calH(g-1,g-1)\]
for $g\geq 2$,
\[\calQ(2(g-k)-4,2k,-1^{2g})\to\calQ(4(g-k)-6,4k+2)\]
for $g\geq 2$ and $0\leq k\leq g-2$,
\[\calQ(2(g-k)-3,2k,-1^{2g+1})\to\calQ(2(g-k)-3,2(g-k)-3,4k+2)\]
 for $g\geq 1$ and $0\leq k\leq g-1$, and
\[\calQ(2(g-k)-3,2k+1,-1^{2g+2})\to\calQ(2(g-k)-3,2(g-k)-3,2k+1,2k+1)\] for $g\geq 1$ and $-1\leq k\leq g-2$, the maps given by this construction are immersions, and the connectedness of genus zero strata, the equality of dimension, and the ergodicity of the geodesic flow give that the images are connected components. We call these connected components the {\em hyperelliptic components} of such strata and denote them by $\hyp(k_{1},\ldots,k_{n})$ and $\Qhyp(k_{1},\ldots,k_{n})$, respectively. The hyperelliptic components contain those hyperelliptic half-translations surfaces $(X,q)$ for which there exists a quadratic differential $q_{0}$ on the sphere such that $\pi^{*}q_{0} = q$. 

\section{Meanders}\label{s:meanders}

An {\em open meander} is a planar arc system consisting of one bi-infinite horizontal line with a second simple arc that intersects the horizontal line transversally. See the left of Figure~\ref{f:openmeanders}.

An {\em open semi-meander} is a planar arc system consisting of one infinite horizontal ray with a second simple arc that intersects the horizontal ray transversally. The arc is allowed to pass around the end-point of the ray. See the right of Figure~\ref{f:openmeanders}.

A {\em closed meander} is a planar curve and arc system consisting of one bi-infinite horizontal line with a second simple closed curve that intersects the horizontal line transversally. See the left of Figure~\ref{f:closedmeanders}.

Finally, a {\em closed semi-meander} is a planar curve and arc system consisting of one infinite horizontal ray and a second simple closed curve that intersects the horizontal ray transversally and is allowed to pass around the end-point of the ray. See the right of Figure~\ref{f:closedmeanders}.

\subsection*{Some terminology} We will call an arc like the upper most arc in the meander on the left of Figure~\ref{f:closedmeanders} a \emph{maximal arc} since the remainder of the meander is contained between its endpoints. We will call a sequence of nested arcs a \emph{rainbow} -- see the upper half of the meander on the left of Figure~\ref{f:closedmeanders} for an example.

\begin{figure}[t]
\begin{center}
\includegraphics[scale=1]{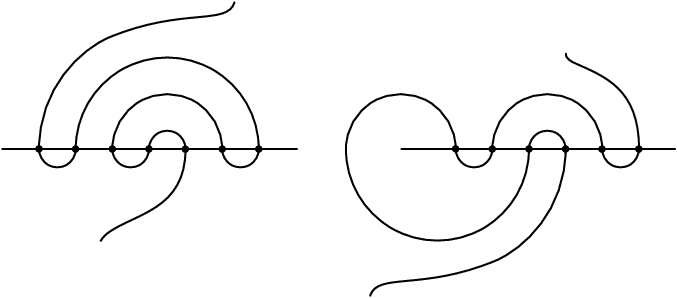}
\end{center}
\caption{An open meander on the left and an open semi-meander on the right.}
\label{f:openmeanders}
\end{figure}

\begin{figure}[t]
\begin{center}
\includegraphics[scale=1]{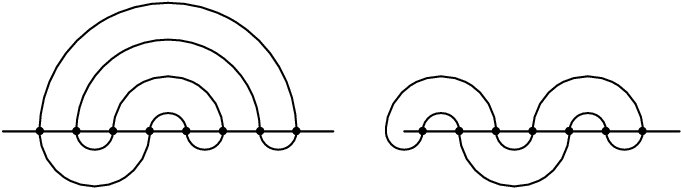}
\end{center}
\caption{A closed meander on the left and a closed semi-meander on the right.}
\label{f:closedmeanders}
\end{figure}

\begin{figure}[t]
\begin{center}
\includegraphics[scale=1]{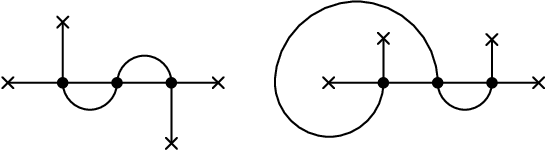}
\end{center}
\caption{A doubly-anchored open meander on the left and a doubly-anchored open semi-meander on the right.}
\label{f:doubly-anchored}
\end{figure}

\begin{figure}[t]
\begin{center}
\includegraphics[scale=0.9]{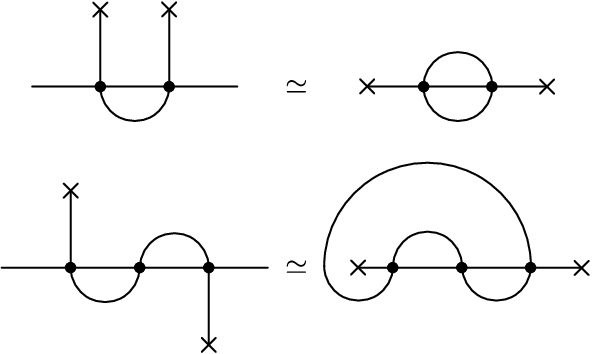}
\end{center}
\caption{Singly-anchored open meanders on the left and the equivalent singly-anchored closed meander and singly-anchored closed semi-meander on the right.}
\label{f:singly-anchored}
\end{figure}

\subsection{Anchored meanders}

Since we will be thinking of meanders as very specific curve and arcs systems on punctured spheres, there will be situations where we want to consider the horizontal lines/rays and infinite transversal arcs as specific arcs between punctures on the sphere and not as infinite lines. In such a case, we will say that the arc, and the associated meander, are {\em anchored}. Conversely, non-anchored infinite lines/arcs will correspond to closed curves on the sphere.

More specifically, we will make the following definitions.

Firstly, suppose that we have an open meander. If we anchor both the horizontal line and the intersecting arc then we will call such a configuration a {\em doubly-anchored open meander}. See Figure~\ref{f:doubly-anchored}. In this situation, both of the lines in the meander are now arcs on the punctured sphere and we observe that the regions that were originally above and below the horizontal line are now connected.

We can similarly define a doubly-anchored open semi-meander.

Now, suppose that we again have an open meander. We will now only anchor the intersecting arc. We will call such a configuration a {\em singly-anchored open meander}. See Figure~\ref{f:singly-anchored}. In this case, the horizontal line remains a closed curve on the punctured sphere while the intersecting arc is genuinely an arc on the punctured sphere. Up to homotopy on the sphere, we can also think of such a configuration as a singly-anchored closed meander or semi-meander where we have anchored the horizontal line/ray.

\subsection{The stratum of a meander}

When we have a closed meander, the curve system is exactly a filling pair on the punctured sphere and we can build the associated $[1,1]$-pillowcase cover. This pillowcase cover will lie in some genus zero quadratic stratum and we will abuse language and notation by calling this the {\em stratum of the meander
}.

Recall that the orders of the zeros of the pillowcase cover associated to the meander are related exactly to the numbers of sides of the complementary polygons. We will carry this idea forward and further abuse our language and notation in order to define the `stratum' of an anchored meander. Though this has no meaning in the sense of stratum that we have used for the moduli spaces considered above, it will be useful language and notation for the following section.

If we think of the anchored meander as a graph in the plane, then we can consider its complementary regions as we did for the closed meanders above. If a complementary region is bounded by edges that are not incident to valence one vertices then we simply treat this complementary region as a polygon, as we did in the closed meander case. If the complementary region is bounded by edges some of which are incident to valence one vertices then we treat this as a polygon where the edges incident to valence one vertices are counted as a single side of the polygon. Once we have constructed these polygons, we associate each polygon to a zero of a specific order that is determined exactly as in the meander case. We also add a zero of order $-1$ for each valence one vertex. This collection of orders of zeros will label a genus zero quadratic stratum and we will call this the {\em stratum of the anchored meander}. See Figure~\ref{f:meanderstrata} for some examples.

\begin{figure}[t]
\begin{center}
\includegraphics[scale=1]{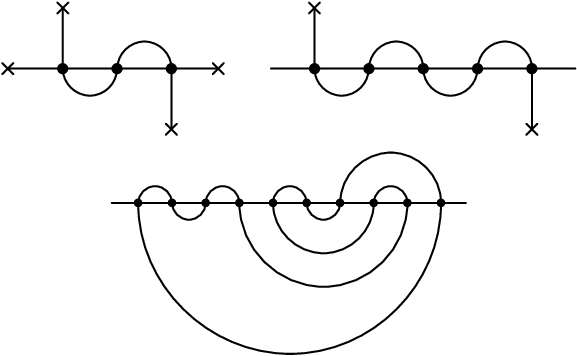}
\end{center}
\caption{A doubly-anchored open meander in $\calQ(2,-1^{6})$, a singly-anchored meander in $\calQ(1,1,-1^{6})$ and a closed meander in $\calQ(2,1,-1^{7})$.}
\label{f:meanderstrata}
\end{figure}

\subsection{Meanders from hyperelliptic {\boldmath $[1,1]$}-pillowcase covers}

The core curves of the cylinders of a hyperelliptic $[1,1]$-pillowcase cover are mapped to themselves by the hyperelliptic involution. Such curves are said to be \emph{symmetric}. A non-separating symmetric curve descends to an arc on the punctured sphere while a symmetric separating curve will descend to a closed curve. Therefore, the core curves of a hyperelliptic $[1,1]$-pillowcase cover will descend to a doubly-anchored open meander, a singly-anchored open meander, or a closed meander, if the number of separating core curves is 0, 1, or 2, respectively. The proofs for hyperelliptic components and non-separating core curves in the previous work of the author~\cite{J0,J1} were essentially using a less formal realisation of this idea.

\section{Minimal constructions}\label{s:constructions}

In this section, we will prove Theorem~\ref{th:main}. We will begin by giving minimal constructions of the meanders and their variations in the genus zero quadratic strata relevant to the constructions of hyperelliptic pillowcase covers. We will then describe how these constructions can be lifted to give minimal constructions of hyperelliptic $[1,1]$-pillowcase covers in the hyperelliptic components of the moduli space of quadratic differentials.

\subsection{Minimal closed meanders}

We begin by constructing closed meanders in the strata relevant to the constructions of hyperelliptic pillowcase covers.

\begin{figure}[t]
\begin{center}
\includegraphics[scale=1]{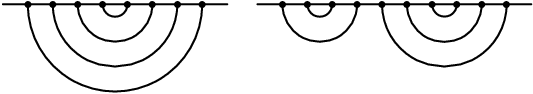}
\end{center}
\caption{Possible configurations for the bottom of the meander.}
\label{f:closedmeanderproof1}
\end{figure}

\begin{figure}[t]
\begin{center}
\includegraphics[scale=0.75]{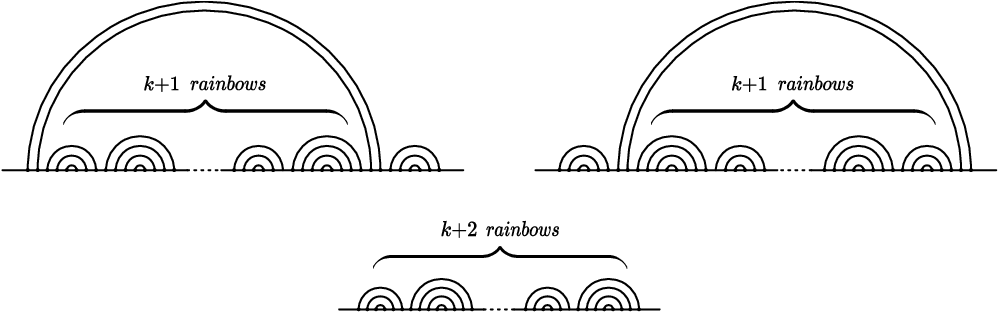}
\end{center}
\caption{Forced configurations for the top of the meander.}
\label{f:closedmeanderproof2}
\end{figure}

\begin{proposition}\label{p:closedmeander1}
The minimum number of crossings required to construct a closed meander in $Q(k,-1^{k+4})$, $k\geq 0$, is $4k+4$.
\end{proposition}

\begin{proof}
We suppose that the zero of order $k$ lies above the horizontal line. Then the bottom of the meander can only contain 4-gons and bigons and so is forced to look like one of the two configurations shown in Figure~\ref{f:closedmeanderproof1}. That is, it either consists of a single rainbow or exactly two rainbows.

Up to cyclic rotation along the horizontal line, we may assume that the configuration on the left holds. That is, it consists of a single rainbow.

The top half of the meander can only contain bigons, 4-gons and a single $(2k+4)$-gon. It must therefore have one of the configurations shown in Figure~\ref{f:closedmeanderproof2}.

We will argue that it must have the bottom configuration consisting of $k+2$ rainbows. Indeed, for minimality of the meander, we cannot have the rightmost (or leftmost) rainbow in the top left (top right) configuration. Indeed, we would be able to ``unwind" this rainbow without changing the stratum of the meander (since we are ignoring zeroes of order zero) yet reducing the number of crossings. See Figure~\ref{f:closedmeanderproof3}. We will no longer be able to unwind once the top of the meander has one of the configurations shown in Figure~\ref{f:closedmeanderproof4}.

\begin{figure}[t]
\begin{center}
\includegraphics[scale=0.8]{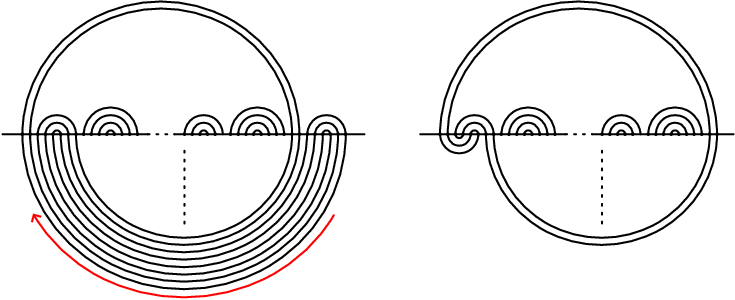}
\end{center}
\caption{Unwinding the rightmost rainbow.}
\label{f:closedmeanderproof3}
\end{figure}

\begin{figure}[t]
\begin{center}
\includegraphics[scale=1]{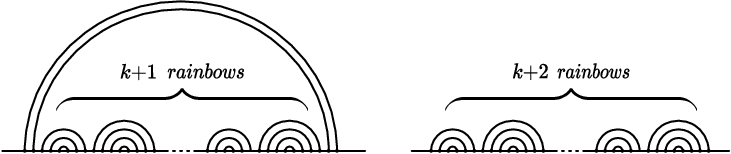}
\end{center}
\caption{Configurations for the top of the meander after unwinding is completed.}
\label{f:closedmeanderproof4}
\end{figure}

However, we cannot have the configuration on the left as this will give us maximal arcs on both the top and the bottom which would close to form a circle and so the meander would be disconnected. We are already assuming that we have a maximal arc on the bottom and so the top must look as in the right of Figure~\ref{f:closedmeanderproof4}.


\begin{claim}
The minimum number of crossings required to construct a closed meander in $Q(k,-1^{k+4})$, $k\geq 0$, is at least $4k+4$.
\end{claim}

\begin{proof}
A connected meander has an even number of arcs (each arc in the meander lies in the bottom, then the top, then the bottom, and so on). Suppose that we have $4k+2$ or fewer arcs in total, then we have at most $2k+1$ arcs in the top of the meander. We have $k+2$ rainbows in the top of the meander and so, $k+2$ of these arcs are used to form the upper arcs of the rainbows. This leaves us with at most $k-1$ additional arcs. Therefore, at least 3 of the $k+2$ rainbows are single-arc rainbows (i.e., just bigons). The configuration of the meander, the bottom being a single rainbow and the top being a sequence of rainbows, means that every connected component of the meander can contain at most two of the single-arc rainbows in the top of the meander. To see this, think of the single-arc rainbows as being `ends' of a connected component of the meander. We start in one end then are forced through the bottom rainbow, then a top rainbow, then the bottom rainbow, and so on, until we hit the bigon at the other end. We have at least 3 single-arc rainbows in the top, so we have at least two connected components, and so the meander is not connected if we have $4k+2$ or fewer arcs in the meander. Therefore, we require at least $4k+4$ arcs in total and so at least $4k+4$ crossings.
\end{proof}

The meander show in Figure~\ref{f:closedmeanderproof7} achieves this lower bound and therefore completes the proof of the proposition. It is realised by having $k+2$ rainbows in the top. The leftmost rainbow is a single-arc bigon, the rightmost rainbow is a 2-arc rainbow, the second to leftmost rainbow is a 2-arc rainbow, the second to rightmost rainbow is a 2-arc rainbow, and we continue inwards alternating from left to right with 2-arc rainbows until we reach the final innermost rainbow which is also a single-arc bigon. So we have $k$ 2-arc rainbows and 2 single-arc rainbows which gives us the $2k+2$ arcs on the top, and so the $4k+4$ crossings in the meander. This construction was known to the author in his thesis~\cite[Proposition 6.21]{J0}, but only conjectured to be minimal in that work.
\end{proof}

\begin{figure}[t]
\begin{center}
\includegraphics[scale=1]{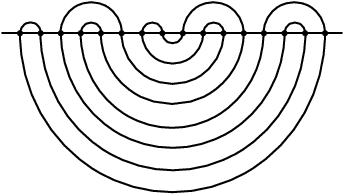}
\end{center}
\caption{A meander realising the $4k+4$ crossing lower bound.}
\label{f:closedmeanderproof7}
\end{figure}

\begin{figure}[b]
\begin{center}
\includegraphics[scale=1]{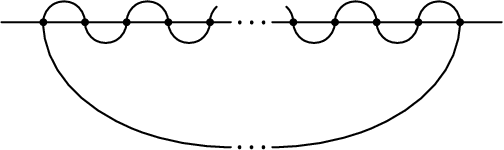}
\end{center}
\caption{The minimal number of crossings ($2k_{1}+4$) for a closed meander in $\calQ(k_{1}^{2},-1^{2k_{1}+4})$.}
\label{f:closedmeanderproof8}
\end{figure}

We now handle strata of the form $Q(k_{1},k_{2},-1^{k_{1}+k_{2}+4})$. The cases $k_{1} = k_{2}$ and $k_{1} = k_{2}+1$ were previously handled in the author's thesis~\cite[Proposition 6.23]{J0}.

\begin{proposition}\label{p:closedmeander2}
The minimum number of crossings required to construct a closed meander in $Q(k_{1},k_{2},-1^{k_{1}+k_{2}+4})$, $k_{1}\geq k_{2}\geq -1$, is $4k_{1}-2k_{2}+4$.
\end{proposition}

\begin{proof}
We first give a construction that realises the claimed minimum number of crossings. We will then prove that this number is indeed minimal.

If $k_{1} = k_{2}$, then the meander in Figure~\ref{f:closedmeanderproof8} realises a closed meander in the stratum $Q(k_{1},k_{1},-1^{2k_{1}+4})$ with $2k_{1}+4 = 4k_{1}-2k_{2}+4$ crossings. If $k_{1} > k_{2}$, we begin by producing a closed meander in $Q(k_{2},k_{2},-1^{2k_{2}+4})$ again using the method in Figure~\ref{f:closedmeanderproof8}. So far, we have $2k_{2}+4$ crossings. Now, we start winding the rightmost upper bigon inside the bottom bigon to its left using the technique of Figure~\ref{f:closedmeanderproof7} until we have added $k_{1}-k_{2}$ additional bigons. See Figure~\ref{f:closedmeanderproof13} which demonstrates the first step of the winding. Each added bigon requires 4 additional crossings, so the final meander has $2k_{2}+4 + 4(k_{1}-k_{2}) = 4k_{1}-2k_{2}+4$ crossings, as desired.

\begin{figure}[t]
\begin{center}
\includegraphics[scale=1]{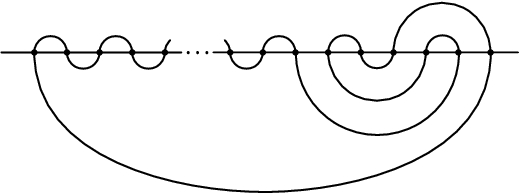}
\end{center}
\caption{A closed meander in $Q(k_{1},k_{2},-1^{k_{1}+k_{2}+4})$ realising $4k_{1}-2k_{2}+4$ crossings.}
\label{f:closedmeanderproof13}
\end{figure}

We now prove that these meanders are minimal. We continue with the assumption that $k_{1}\geq k_{2}\geq -1$ and we will use proof by induction on $k_{1}$.

Our base case is $k_{1} = k_{2}$. In this case the minimum number of crossings is realised by the configuration shown in Figure~\ref{f:closedmeanderproof8}. This is clear since, given a fixed number of crossings, this configuration maximises the number of rainbows (in this case bigons) on each side of the meander. Therefore, the minimum number of crossings for a closed meander in $Q(k_{1},k_{2},-1^{k_{1}+k_{2}+4})$ with $k_{1} = k_{2}$ is $4k_{1}-2k_{2}+4$.

Now suppose that a closed meander in $Q(k_{1},k_{2},-1^{k_{1}+k_{2}+4})$, $k_{1}\geq k_{2}\geq -1$, is minimised with $4k_{1}-2k_{2}+4$ crossings. Consider a minimising closed meander in $Q(k_{1}+1,k_{2},-1^{k_{1}+k_{2}+4})$. A similar `unwinding' process to that used in the proof of Proposition~\ref{p:closedmeander1} (Figure~\ref{f:closedmeanderproof3}) can be used to show that such a minimal closed meander must have one zero on the top of the meander, say the zero of order $k_{1}+1$, and the other, so the zero of order $k_{2}$, on the bottom.

This means that there must be $k_{1}+3$ rainbows (up to cyclic rotation) in the top and $k_{2}+2$ rainbows (up to cyclic rotation) in the bottom. Since $k_{1}\geq k_{2}$, we have $k_{1}+3 > k_{2}+2$. There must therefore be at least one rainbow in the bottom of the meander that is connected to at least three of the rainbows in the top of the meander.

Consider the centre of this bottom rainbow. Up to cyclic rotation of the meander, it looks as in the left of Figure~\ref{f:closedmeanderproof1}. We observe that, if the meander is minimal (in terms of the number of crossings), then both ends of the central arc of the bottom rainbow cannot be contained in the same rainbow in the top of the meander, as in Figure~\ref{f:closedmeanderproof14}. This is because such a configuration arises from a (sequence of) winding move(s) in which either the number of rainbows is not changed or the number of rainbows is decreased. This can be seen by unwinding and seeing that the number of rainbows is unchanged or increased while reducing the number of crossings. In either case, the original meander was not minimal.

\begin{figure}[t]
\begin{center}
\includegraphics[scale=1]{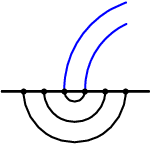}
\end{center}
\caption{Impossible configuration if the meander is minimal.}
\label{f:closedmeanderproof14}
\end{figure}

Therefore, we have two distinct rainbows attached to this central arc in the bottom rainbow. See Figure~\ref{f:closedmeanderproof15}. Since the bottom rainbow connects to at least three rainbows in the top, one of these rainbows, say the blue one, is closed on the bottom rainbow.

If the blue rainbow is a single bigon and the purple rainbow has at least three arcs, then we can reduce the number of crossings and preserve the stratum; that is, the initial meander was not minimal. This is done via the unwinding shown in Figure~\ref{f:closedmeanderproof16}. So if the meander is minimal then the purple rainbow must have two arcs and the same unwinding process (Figure~\ref{f:closedmeanderproof17}) allows us to reduce the zero of order $k_{1}+1$ to a zero of order $k_{1}$ while keeping the zero of order $k_{2}$. In the process, we lose 4 intersections. By the inductive hypothesis, the minimum number of crossings for the resulting meander is $4k_{1}-2k_{2}+4$ and so our original meander had at least $4k_{1}-2k_{2}+4+4 = 4(k_{1}+1)-2k_{2}+4$ crossings. This is realised by the construction above and so, if such a configuration always exists, the minimum number of crossings for a meander in $Q(k_{1}+1,k_{2},-1^{k_{1}+k_{2}+4})$ is indeed $4(k_{1}+1)-2k_{2}+4$, as required.

\begin{figure}[t]
\begin{center}
\includegraphics[scale=1]{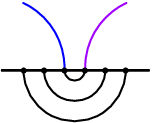}
\end{center}
\caption{Forced configuration at centre of the bottom rainbow.}
\label{f:closedmeanderproof15}
\end{figure}

\begin{figure}[b]
\begin{center}
\includegraphics[scale=1]{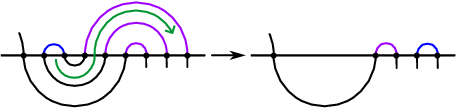}
\end{center}
\caption{Stratum-preserving unwinding.}
\label{f:closedmeanderproof16}
\end{figure}

\begin{figure}[t]
\begin{center}
\includegraphics[scale=1]{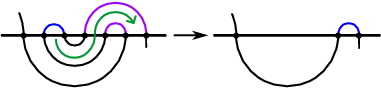}
\end{center}
\caption{Stratum-reducing unwinding. The zero of order $k_{1}+1$ reduces to a zero of order $k_{1}$.}
\label{f:closedmeanderproof17}
\end{figure}

We will now rule out the possibility that the blue rainbow can have more than a single arc. In each case, we will argue that the initial meander was not minimal. So, suppose that the blue rainbow has at least two arcs and that the purple rainbow has a strictly greater number of arcs (if it has strictly fewer, then it must also be closed on the bottom rainbow and symmetric arguments to the below apply, and if they have the same number of arcs then they will form a disjoint connected component of the initial meander). Suppose that the central arc of the purple rainbow lies strictly to the left of the red arc in Figure~\ref{f:closedmeanderproof18}. Note that this red arc does exist since the bottom rainbow connects to at least three rainbows in the top. We can then perform a stratum-preserving unwinding as follows. Firstly, unwind the bigon corresponding to the central arc of the purple rainbow as shown in Figure~\ref{f:closedmeanderproof18}. Finally, unwind the resulting sub-rainbow of the blue rainbow that lies to the left of the unwound purple central arc. This reduces the number of intersections but preserves the stratum and so the initial meander was not minimal.

\begin{figure}[t]
\begin{center}
\includegraphics[scale=1]{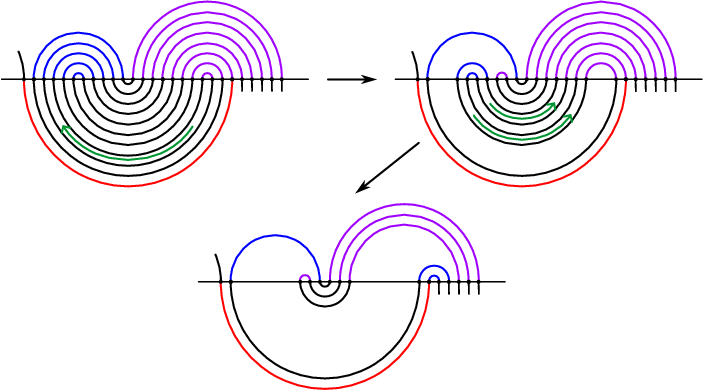}
\end{center}
\caption{Stratum-preserving unwinding if the blue rainbow is not a bigon and the centre of the purple rainbow lies to the left of the red arc.}
\label{f:closedmeanderproof18}
\end{figure}

\begin{figure}[t]
\begin{center}
\includegraphics[scale=1]{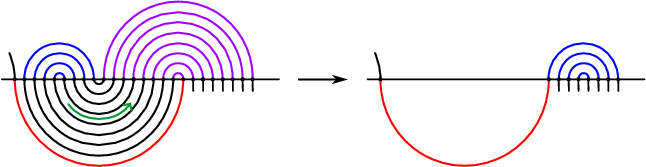}
\end{center}
\caption{Stratum-reducing unwinding if the blue rainbow is not a bigon.}
\label{f:closedmeanderproof19}
\end{figure}

Now suppose that the central arc of the purple rainbow has its right endpoint on the red arc. In this situation, we can unwind and reduce the order of the top zero from $k_{1}+1$ to $k_{1}$. See Figure~\ref{f:closedmeanderproof19}. Since the blue rainbow has at least two arcs, this reduces the number of crossings by at least 8. The inductive hypothesis gives us that the resulting meander has at least $4k_{1}-2k_{1}+4$ crossings so our original meander had at least $4k_{1}-2k_{1}+4 + 8 = 4(k_{1}+1)-2k_{1}+4+4$ which is not minimal. Finally, if the left endpoint of the central arc of the purple rainbow lies on or to the right of the red arc then a stratum-preserving unwinding can again be performed and so the initial meander was not minimal. See Figure~\ref{f:closedmeanderproof20}.

\begin{figure}[t]
\begin{center}
\includegraphics[scale=1]{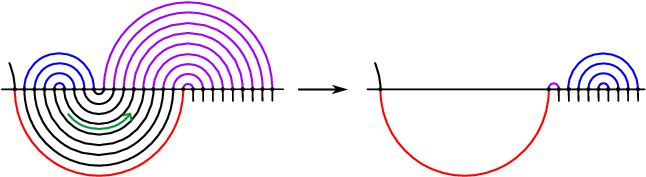}
\end{center}
\caption{Stratum-preserving unwinding if the blue rainbow is not a bigon and the centre of the purple rainbow lies to the right of the red arc.}
\label{f:closedmeanderproof20}
\end{figure}

To conclude, if the meander in stratum $Q(k_{1}+1,k_{2},-1^{k_{1}+k_{2}+4})$ is minimal, then it must contain the configuration shown in Figure~\ref{f:closedmeanderproof17} in which case, by induction, it has at least $4(k_{1}+1)-2k_{2}+4$ crossings which is realisable and so minimal.
\end{proof}

\subsection{Minimal singly-anchored open meanders}

Here, we handle singly-anchored open meanders.

\begin{proposition}\label{p:saom}
The minimum number of crossings required to construct a singly-anchored open meander in $Q(k_{1},k_{2},-1^{k_{1}+k_{2}+4})$, $k_{1}\geq k_{2}\geq -1$, is $\max\{2k_{1}+2,2k_{2}+3\}$.
\end{proposition}

\begin{figure}[b]
\begin{center}
\includegraphics[scale=0.9]{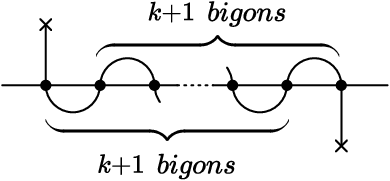}
\end{center}
\caption{Minimally constructing a singly-anchored open meander in $\calQ(k,k,-1^{2k+4})$.}
\label{f:saproof1}
\end{figure}

\begin{proof}
It is easily checked that the configuration in Figure~\ref{f:saproof1} minimally realises a singly-anchored open meander in the stratum $\calQ(k,k,-1^{2k+4})$ using $2k+3$ crossings. Now suppose that $k_{1}>k_{2}$. It can then be argued almost exactly as in the case of Proposition~\ref{p:closedmeander1} above that the winding procedure shown in Figure~\ref{f:saproof2} minimally realises a singly-anchored open meander in $\calQ(k_{1},k_{2},-1^{k_{1}+k_{2}+4})$. We see that we have to add $2(k_{1}-k_{2}-1)+1 = 2k_{1}-2k_{2}-1$ crossings to achieve this. In total, we have used $2k_{1}-2k_{2}-1+2k_{2}+3 = 2k_{1}+2$ crossings. Since we are assuming that $k_{1}\geq k_{2}$, we have that $2k_{2}+3 \geq 2k_{1}+2$ if and only if $k_{1}=k_{2}$ in which case $2k_{2}+3 > 2k_{1}+2$. Otherwise, $k_{1}>k_{2}$ and $2k_{1}+2 > 2k_{2}+3$. The claim of the proposition follows.
\end{proof}

\begin{figure}[t]
\begin{center}
\includegraphics[scale=1]{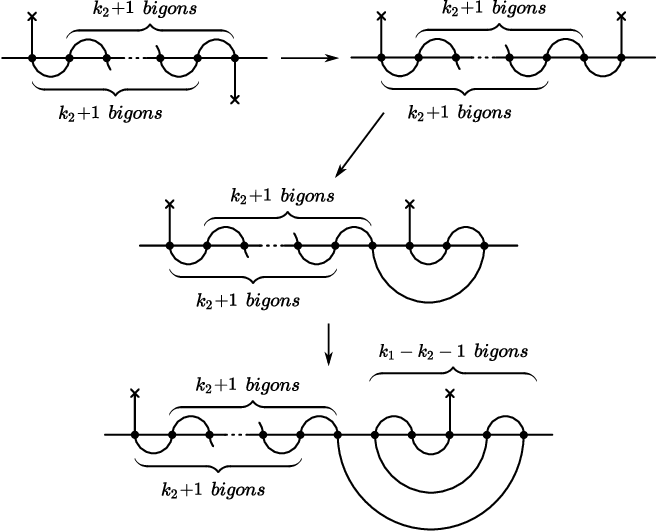}
\end{center}
\caption{Winding to achieve a singly-anchored open meander in $\calQ(k_{1},k_{2},-1^{k_{1}+k_{2}+4})$.}
\label{f:saproof2}
\end{figure}

\subsection{Minimal doubly-anchored open meanders and semi-meanders}

Here, we handle doubly-anchored open meanders and semi-meanders.

\begin{proposition}\label{p:daom}
The minimum number of crossings required to construct a doubly-anchored open meander in $Q(k,-1^{k+4})$, $k\geq 0$, is $k+1$.
\end{proposition}

\begin{proof}
It is clear that the minimal construction of a doubly-anchored open meander in $\calQ(k,-1^{k+4})$ is given by the configuration shown in Figure~\ref{f:daproof1}. Indeed, this construction maximises the number of bigons added per crossing. We see that we require $k$ bigons for a zero of order $k$ which can be optimally achieved by $k+1$ crossings.
\end{proof}

\begin{figure}[t]
\begin{center}
\includegraphics[scale=0.9]{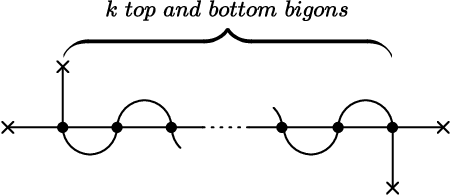}
\end{center}
\caption{Minimal construction of a doubly-anchored open meander in $\calQ(k,-1^{k+4})$.}
\label{f:daproof1}
\end{figure}

\begin{proposition}\label{p:daosm}
The minimum number of crossings required to construct a doubly-anchored open semi-meander in $Q(k_{1},k_{2},-1^{k_{1}+k_{2}+4})$, $k_{1}\geq k_{2}\geq 0$, is $k_{1}+k_{2}+2$.
\end{proposition}

\begin{proof}
Similar to the proof of the previous proposition, it can be checked that the minimal construction of a doubly-anchored open semi-meander in the stratum $\calQ(k_{1},k_{2},-1^{k_{1}+k_{2}+4})$ is given by the configuration shown in Figure~\ref{f:daproof2}. We see that this construction requires 2 crossings to set up the initial semi-meander, then $k_{1}$ additional crossing to produce the $k_{1}$ bigons and $k_{2}$ additional crossings to produce the other $k_{2}$ bigons.
\end{proof}

\begin{figure}[t]
\begin{center}
\includegraphics[scale=1]{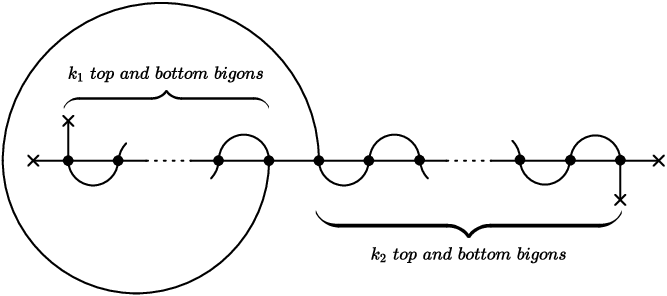}
\end{center}
\caption{Minimal construction of a doubly-anchored open semi-meander in $\calQ(k_{1},k_{2},-1^{k_{1}+k_{2}+4})$.}
\label{f:daproof2}
\end{figure}

\subsection{Lifting to hyperelliptic connected components}

Here, we will give the final propositions that complete the proof of Theorem~\ref{th:main}.

Note that the non-separating core curve cases of the following propositions were already proved by the author~\cite{J0,J1} but we reprove them here in the formalised setting of meander lifting.

We first require a technical lemma concerning lifts of curves via the covering map associated to a hyperelliptic involution. It is well-known (for example, it is used in work of Brendle-Margalit-Putman~\cite{BMP} on the hyperelliptic Torelli group), but we include a proof here for completeness.

\begin{figure}[t]
\begin{center}
\includegraphics[scale=0.9]{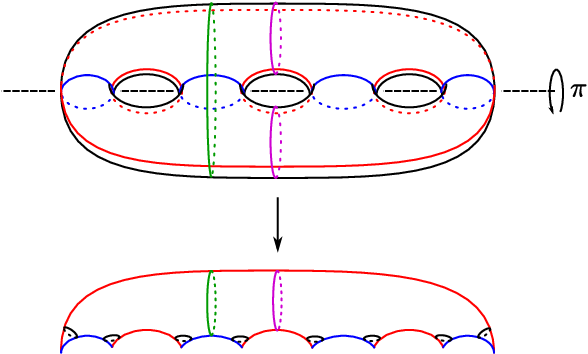}
\end{center}
\caption{Symmetric lifts of curves via a hyperelliptic involution.}
\label{f:lemma}
\end{figure}

\begin{lemma}\label{l:lift}
Let $\tau:S_{g}\to S_{0,2g+2}$ be the double cover associated to a hyperelliptic involution $\iota:S_{g}\to S_{g}$. A closed curve on the punctured sphere $S_{0,2g+2}$ lifts through $\tau$ to a symmetric separating curve on $S_{g}$ if and only if it encloses an odd number of punctures.
\end{lemma}

\begin{proof}
$\Leftarrow$: Let $\gamma$ be a simple closed curve on $S_{0,2g+2}$ that encloses an odd number of punctures. Up to change of coordinates, we can take $\gamma$ to be a `standard' curve that the separates the punctures into $2k+1$ on one side and $2g+1-2k$ on the other. See the green curve in the bottom of Figure~\ref{f:lemma} for an example. This lifts under $\tau$ to a symmetric separating curve separating a genus $k$ surface on one side and a genus $g-k$ surface on the other. Indeed, observe that in Figure~\ref{f:lemma} $\gamma$ makes an odd number of intersections with the blue arcs on the sphere whose images upstairs separate the two sheets of the cover. So a single `unfolding' will not be a closed curve. By the Birman-Hilden correspondence~\cite{BH1,BH2}, the change of coordinates homeomorphism on $S_{0,2g+2}$ lifts to a homeomorphism of $S_{g}$ that commutes with the hyperelliptic involution, and so the argument holds for any curve $\gamma$ enclosing an odd number of punctures.

$\Rightarrow$: We prove the contrapositive. Take a simple closed curve $\gamma$ on $S_{0,2g+2}$ that encloses an even number of punctures. Up to change of coordinates, we can take it to be the standard curve that separates say $2k$ and $2g+2-2k$ on either side. This lifts through $\tau$ to a pair of non-separating curves that are symmetric to one another under $\iota$. Indeed, we have an even number of intersections with the blue arcs now. Again, the same Birman-Hilden correspondence argument holds.
\end{proof}

\begin{figure}[t]
\begin{center}
\includegraphics[scale=0.85]{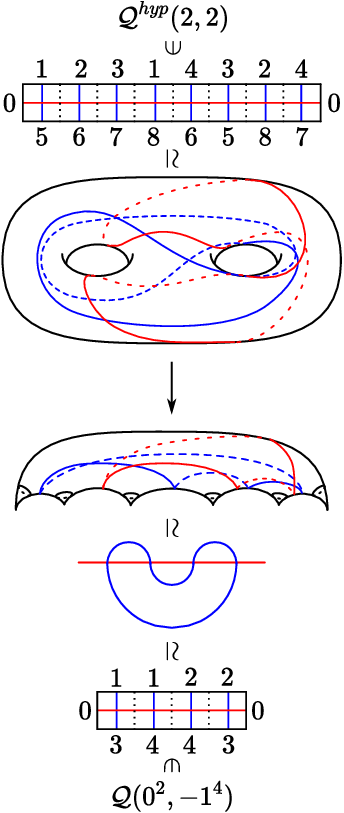}
\end{center}
\caption{An example of lifting a meander in $\calQ(0^{2},-1^{4})$ to a $[1,1]$-pillowcase cover in $\Qhyp(2,2)$ where both core curves are separating.}
\label{f:lift-ex}
\end{figure}

\begin{remark}
We note that the above lemma also holds for the case of a twice punctured surface whose punctures are symmetric under the hyperelliptic involution. In this case, the curve must separate into two sets of odd size those punctures that are the images of the Weierstrass points. This version is relevant to the constructions in hyperelliptic components with two poles.
\end{remark}

An example of the lifting process used to produce a $[1,1]$-pillowcase cover in $\Qhyp(2,2)$ with both core curves separating is shown in Figure~\ref{f:lift-ex}. Recall that a hyperelliptic $[1,1]$-pillowcase cover in $\Qhyp(2,2)$ is a double cover of a $[1,1]$-pillowcase cover in $\calQ(0^{2},-1^{4})$. The minimal construction of a closed meander in this stratum is given by a closed meander with four crossings which lifts to a filling pair with eight crossings. Hence, we see that the minimum number of squares required for a $[1,1]$-pillowcase cover with both core curves separating in $\Qhyp(2,2)$ is eight.

\begin{proposition}\label{p:lift1}
Let $g\geq 2$. The minimal number of squares required for a $[1,1]$-pillowcase cover in the hyperelliptic connected components $\hyp(2g-2)$ and $\hyp(g-1,g-1)$ are $4g-4$ and $4g-2$, respectively. Both of the core curves are forced to be non-separating in this setting.
\end{proposition}

\begin{proof}
Since both of the core curves are non-separating they descend under the hyperelliptic involution to two arcs on the punctured sphere. As discussed in Subsection~\ref{s:hypcomp}, the hyperelliptic components $\hyp(2g-2)$ and $\hyp(g-1,g-1)$ are realised by lifting quadratic differentials in $\calQ(2g-3,-1^{2g+1})$ and $\calQ(2g-2,-1^{2g+2})$, respectively. This in turn means that that the arcs on the punctured sphere form doubly-anchored open meanders in the strata $\calQ(2g-3,-1^{2g+1})$ and $\calQ(2g-2,-1^{2g+2})$, respectively. Note that, since all the side identifications for a pillowcase cover in $\hyp(2g-2)$ or $\hyp(g-1,g-1)$ are made by translation, the core curves do not descend to a doubly-anchored open semi-meander. By Proposition~\ref{p:daom}, a doubly-anchored open meander in the stratum $\calQ(k,-1^{k+4})$, $k\geq 0$, requires at least $k+1$ crossings. So a doubly-anchored open meander in $\calQ(2g-3,-1^{2g+1})$ requires at least $2g-2$ crossings and a doubly-anchored meander in $\calQ(2g-2,-1^{2g+2})$ requires at least $2g-1$ crossings. Therefore, the number of intersections between the core curves (which are lifts by a double cover) are $4g-4$ for $\hyp(2g-2)$ and $4g-2$ for $\hyp(g-1,g-1)$. Since the number of squares of the $[1,1]$-pillowcase covers are at least the number of intersections of the core curves, we are done.
\end{proof}

\begin{proposition}\label{p:lift2}
Let $k\geq j\geq 0$. The minimal number of squares required for a $[1,1]$-pillowcase cover in the hyperelliptic connected component $\Qhyp(4j+2,4k+2)$ is $4j+4k+4$ if both of the core curves are non-separating, $\max\{8j+6,8k+4\}$ if only one of the core curves is separating, and $16k-8j+8$ if both of the core curves are separating.
\end{proposition}

\begin{proof}
If both of the core curves are non-separating then they descend under the hyperelliptic involution to a doubly-anchored open semi-meander in the stratum $\calQ(2j,2k,-1^{2j+2k+4})$. Proposition~\ref{p:daosm} gives us that such a doubly-anchored open semi-meander requires at least $2j+2k+2$ crossings. Therefore, after lifting, the $[1,1]$-pillowcase cover in $\Qhyp(4j+2,4k+2)$ requires at least $4j+4k+4$ squares, as claimed.

If only one of the core curves is separating, then the core curves descend to a singly-anchored open meander in the stratum $\calQ(2j,2k,-1^{2j+2k+4})$. By Proposition~\ref{p:saom}, such a singly-anchored open meander requires at least $\max\{4j+3,4k+2\}$ crossings. Therefore, the $[1,1]$-pillowcase cover in $\Qhyp(4j+2,4k+2)$ requires at least $\max\{8j+6,8k+4\}$ squares, as claimed. Note that in our constructions of minimal singly-anchored open meanders in the proof of Proposition~\ref{p:saom}, and once all the punctures have been added, the closed horizontal curve separates the the punctures into two sets of odd size and so will indeed lift two a closed curve.

Finally, if both of the core curves are separating then they will descend to a closed meander in the stratum $\calQ(2j,2k,-1^{2j+2k+4})$. By Proposition~\ref{p:closedmeander2}, such meander requires at least $8k - 4j + 4$ crossings. Therefore, the $[1,1]$-pillowcase cover in $\Qhyp(4j+2,4k+2)$ requires at least $16k - 8j + 8$ squares, as claimed. Again, it can be checked that both of the curves forming the meander separate the punctures into two sets of odd size and so do lift to closed curves.
\end{proof}

The proofs of the following three propositions are analogous.

\begin{proposition}\label{p:lift3}
Let $j\geq 1$ and $k\geq 0$ with $j\geq k$. The minimal number of squares required for a $[1,1]$-pillowcase cover in the hyperelliptic connected component $\Qhyp(4j+2,2k-1,2k-1)$ is $4j+4k+2$ if both of the core curves are non-separating, $8j+4$ if only one of the core curves is separating, and $16j-8k+12$ if both of the core curves are separating.
\end{proposition}


\begin{proposition}\label{p:lift4}
Let $k > j \geq 0$. The minimal number of squares required for a $[1,1]$-pillowcase cover in the hyperelliptic connected component $\Qhyp(4j+2,2k-1,2k-1)$ is $4j+4k+2$ if both of the core curves are non-separating, $8k$ if only one of the core curves is separating, and $16k-8j$ if both of the core curves are separating.
\end{proposition}


\begin{proposition}\label{p:lift5}
Let $k\geq 1$ and $j\geq 0$ with $k\geq j$. The minimal number of squares required for a $[1,1]$-pillowcase cover in the hyperelliptic connected component $\Qhyp(2j-1,2j-1,2k-1,2k-1)$ is $4j+4k$ if both of the core curves are non-separating, $\max\{8j+2,8k\}$ if only one of the core curves is separating, and $16k-8j+4$ if both of the core curves are separating.
\end{proposition}


Finally, we handle the case of $\Qhyp(2,-1,-1)$.

\begin{proposition}\label{p:lift6}
The minimal number of squares required for a $[1,1]$-pillowcase cover in the hyperelliptic connected component $\Qhyp(2,-1,-1)$ is 3 if both of the core curves are non-separating, 4 if only one of the core curves is separating, and 12 if both of the core curves are separating.
\end{proposition}

\begin{proof}
The doubly-anchored open meander in $\calQ(0,-1^{4})$ given by Proposition~\ref{p:daom} cannot be lifted to the stratum $\Qhyp(2,-1,-1)$ as all of the poles, being anchors on the sphere, are images of Weierstrass points. The doubly-anchored open semi-meander given by Proposition~\ref{p:daosm} uses two crossings and lifts to a $[1,1]$-pillowcase cover with 4 squares. However, there is a doubly-anchored open semi-meander in $\calQ(0,-1^{4})$, as shown in Figure~\ref{f:lift6}, that reuses an anchor for both the vertical and horizontal arc. This anchor crossing is not doubled under the lifting procedure as it is the image of a Weierstrass point. Hence the minimal number of squares required for a $[1,1]$-pillowcase cover in the hyperelliptic component $\Qhyp(2,-1,-1)$ is 3. The remaining cases follow as above. Proposition~\ref{p:saom} requires a minimum of 2 crossings for a singly-anchored open meander in $\calQ(0,-1^{4})$ which lifts to a $[1,1]$-pillowcase cover in $\Qhyp(2,-1,-1)$ with at least 4 squares (one core curve separating), and Proposition~\ref{p:closedmeander2} requires at least $6$ crossings for a closed meander in $\calQ(0,-1^{4})$ which lifts to a $[1,1]$-pillowcase cover in $\Qhyp(2,-1,-1)$ with at least 12 squares and both core curves being separating.
\end{proof}

\begin{figure}[t]
\begin{center}
\includegraphics[scale=0.9]{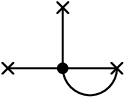}
\end{center}
\caption{A doubly-anchored semi-meander in $\calQ(0,-1^{4})$ using 2 crossings one of which is located at an anchor.}
\label{f:lift6}
\end{figure}

We see that the values in the above propositions are exactly those stated in Table~\ref{tab:theorem} of Theorem~\ref{th:main}. Hence, the proof of Theorem~\ref{th:main} is complete.

\section{Ratio-optimisers in the Johnson filtration}\label{s:ROs}

In this section, we prove Theorem~\ref{th:RO}.

\subsection{Ratio-optimising pseudo-Anosovs}

Recall that the {\em \teichmuller space $\T(S)$} of a closed surface $S$ of genus $g\geq 2$ is the set of equivalence classes of pairs $(X,\varphi)$, where $X$ is a hyperbolic surface of genus $g$ and $\varphi:S\to X$ is a homeomorphism. Two pairs $(X_{1},\varphi_{1})$ and $(X_{2},\varphi_{2})$ are equivalent if the change of marking map $\varphi_{2}\circ\varphi_{1}^{-1}$ is isotopic to an isometry. As such, the length of an isotopy class of a simple closed curve for a given point $[(X,\varphi)]\in\T(S)$ is well-defined. \teichmuller space carries a metric $d_{\T}$ called the \teichmuller metric and we denote by $\T(S)$ the metric space $(\T(S),d_{\T})$. Given a pseudo-Anosov homeomorphism $f$, we define the {\em translation length of $f$ on $\T(S)$} to be $\ell_{\T}(f):=\frac{1}{2}\log(K_{f})$, where $K_{f}$ is the dilatation of $f$.

The {\em curve graph $\C(S)$} of the surface $S$, introduced by Harvey~\cite{H}, is defined to be the graph whose vertices are isotopy classes of essential simple closed curves on the surface $S$, with two vertices joined by an edge if and only if they can be realised disjointly on $S$. Assigning length 1 to each edge, we equip $\C(S)$ with the associated path metric $d_{\C}$ and denote by $\C(S)$ the metric space $(\C(S),d_{\C})$. Given a pseudo-Anosov homeomorphism $f$, we define the {\em asymptotic translation length of $f$ on $\C(S)$} to be
\[\ell_{\C}(f):=\liminf_{n\to \infty}\frac{d_{\C}(f^{n}(\bs\alpha),\bs\alpha)}{n},\]
for any $\bs\alpha\in\C^{0}(S)$, which for a pseudo-Anosov is a strictly positive limit (see the works of Bowditch~\cite{Bow} and Masur-Minsky~\cite{MM}).

The $\SO(2,\R)\!\setminus\SL(2,\R)$-orbit of a quadratic differential $(X,q)$ gives rise to an embedding of $\Hbb$ into $\T(S)$. We call the image of this embedding the {\em \teichmuller disk} of the quadratic differential. See the work of Herrlich-Schmith{\"u}sen~\cite{HS} for more details.

The {\em systole map} $\sys: \T(S)\to \C(S)$ is defined to be the map that sends a point $x\in\T(S)$ to the isotopy class of the curve with shortest length in the hyperbolic metric determined by $x$ -- the systole. Note that this map is only coarsely well-defined. Indeed, there is not necessarily a unique systole on a given surface. However, the isotopy classes of all of the systoles on a surface form a set in the curve graph of diameter at most two. The study of this map played a key role in the work of Masur-Minsky in which they proved that the curve complex is $\delta$-hyperbolic \cite{MM}. They showed in particular that the map is coarsely $K$-Lipschitz. That is, there exist $K>0$ and $C\geq 0$ such that
\[d_{\C}(\sys(x),\sys(y))\leq K\cdot d_{\T}(x,y)+C,\]
for all $x,y\in\T(S)$.

It is natural to ask what is the optimum Lipschitz constant, $\kappa_{g}$, defined by
\[\kappa_{g}:=\inf\{K>0\,|\,\exists\,C\geq 0\,\text{such that sys is coarsely K-Lipschitz}\},\]
and Gadre-Hironaka-Kent-Leininger determined that the ratio of $\kappa_{g}$ to $1/\log(g)$ is bounded from above and below by two positive constants {\cite[Theorem 1.1]{GHKL}}. To find an upper bound for $\kappa_{g}$, Gadre-Hironaka-Kent-Leininger gave a careful version of the proof of Masur-Minsky that $\sys$ is coarsely Lipschitz. They then constructed pseudo-Anosov homeomorphisms for which the ratio $\ell_{\C}(f)/\ell_{\T}(f)\asymp 1/\log(g)$, where $\ell_{\C}(f)$ and $\ell_{\T}(f)$ are the asymptotic translation lengths of $f$ in $\C(S)$ and $\T(S)$, respectively. A lower bound for $\kappa_{g}$ then followed by noting that, for any pseudo-Anosov homeomorphism $f$, we have
\[\kappa_{g}\geq\frac{\ell_{\C}(f)}{\ell_{\T}(f)}.\]
Pseudo-Anosov homeomorphisms that maximise this ratio are said to be {\em ratio-optimising}.

\subsection{The Johnson filtration}\label{ss:Jfilt}

If $S$ is a surface of genus $g\geq 2$ with $p \leq 1$ punctures we define the mapping class group of $S$, denoted $\Mod(S)$, to be the group of orientation-preserving self-homeomorphisms of $S$ fixing the punctures pointwise up to isotopy relative to the punctures. Let $\Gamma = \pi_{1}(S)$ and let $\Gamma_{i}$ be the $i^{th}$ term of its lower central series. So $\Gamma_{1} = \Gamma$ and, for $i\geq 1$, $\Gamma_{i+1} = [\Gamma,\Gamma_{i}]$. The action of $\Mod(S)$ on $\Gamma$ preserves $\Gamma_{i}$ and so there is a well-defined action of $\Mod(S)$ on $\Gamma/\Gamma_{i}$. Johnson~\cite{Jo1} defined a filtration of the mapping class group, now called the {\em Johnson filtration}, where the $i^{th}$ term of the filtration denoted by $\mathcal{I}^{i}(S)$ is the kernel of the action of $\Mod(S)$ on $\Gamma/\Gamma_{i+1}$. In particular, $\mathcal{I}^{0}(S)=\Mod(S)$, $\mathcal{I}^{1}(S)=\mathcal{I}(S)$ is the well-studied Torelli group, and $\mathcal{I}^{2}(S)=\mathcal{K}(S)$ is the Johnson kernel. The textbook of Farb-Margalit~\cite[Chapter 6]{FM} contains more details on this filtration.

\subsection{The Aougab-Taylor construction}

Using a Thurston construction on filling pairs, Aougab-Taylor~\cite{AT} constructed a large family of pseudo-Anosov homeomorphisms for which $\tau(f) := \ell_{\T}(f)/\ell_{\C}(f)$ was bounded above by a function $F(g)\asymp\log(g)$ {\cite[Theorem 1.1]{AT}}. Moreover, they proved that there exists a \teichmuller disk $\D\cong\Hbb\subset\T(S)$ such that there are infinitely many conjugacy classes of primitive ratio-optimising pseudo-Anosovs $f$ with the invariant axis of $f$ being contained in $\D$. Recall that a group element $g$ of a group $G$ is said to be {\em primitive} if there does not exist a $h\in G$ such that $g = h^{k}$ for $|k|>1$.

When both of the curves in the filling pair are separating curves with geometric intersection number bounded polynomially in the genus, Aougab-Taylor showed that there exist ratio-optimising pseudo-Anosov homeomorphisms lying arbitrarily deep in the Johnson filtration of the mapping class group of the surface~{\cite[Theorem 5.1]{AT}}.

It is natural to ask which connected components of the moduli space of quadratic differentials contain quadratic differentials whose \teichmuller disks are stabilised by such ratio-optimising pseudo-Anosovs lying arbitrarily deep in the \linebreak Johnson filtration of the mapping class group. This was asked by the author in his thesis~\cite[Question 7.4]{J0}.

Theorem~\ref{th:RO} answers this question in the affirmative for hyperelliptic connected components.

\subsection{Proof of Theorem~\ref{th:RO}} 

From the discussion above, we see that it is sufficient to prove that all of the separating filling pairs we obtain from our hyperelliptic $[1,1]$-pillowcase covers have intersection numbers polynomial in the genus of the surfaces. Equivalently, we must show that the hyperelliptic $[1,1]$-pillowcase covers with both core curves being separating are constructed from a number of squares that is polynomial in the genus. This follows from Theorem~\ref{th:main}. Indeed, we see that the greatest number of squares required for such a pillowcase cover in any hyperelliptic connected component of a fixed genus $g$ is the $16g+4$ squares required for $\Qhyp(2g-1,2g-1,-1,-1)$.

\begin{remark}
We have restricted to the case of connected components with no poles as the definition of a Johnson filtration is more subtle for surfaces with at least two punctures. If one uses the partitioned surface definition of the Torelli group considered by Putman~\cite{Put}, then it can be checked that our ratio-optimising pseudo-Anosovs lie in this Torelli group for either of the two possible definitions for a twice-punctured surface. Here there are two partitions of the punctures $\{\{p_{1},p_{2}\}\}$ and $\{\{p_{1}\},\{p_{2}\}\}$. It can be checked that the separating filling pairs obtained from the $[1,1]$-pillowcase covers in $\Qhyp(4g-2,-1,-1)$ or $\Qhyp(2g-1,2g-1,-1,-1)$ are $P$-separating in the sense of Putman for either choice of partition $P$. Since Putman shows that the Torelli group in this case is generated by $P$-separating twists and $P$-bounding pair maps, we see that our ratio-optimisers do indeed lie in the Torelli group for this definition. Putman does not make a partitioned surface definition of the full filtration.
\end{remark}

\section{Other minimal meanders}\label{s:othermeanders}

Here, we discuss minimal closed meanders in more general strata. The method we will use comes from the author's thesis. The new results above allow us to establish more concrete questions for the general case.

\begin{figure}[b]
\begin{center}
\includegraphics[scale=1]{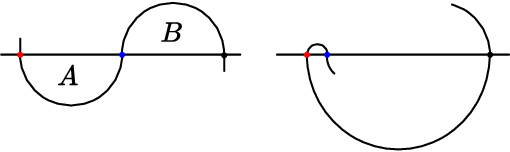}
\end{center}
\caption{The adjacent bigons, and maximal arc with adjacent bigons configuration.}
\label{f:adjacentbigons}
\end{figure}

\begin{figure}[t]
\begin{center}
\includegraphics[scale=1]{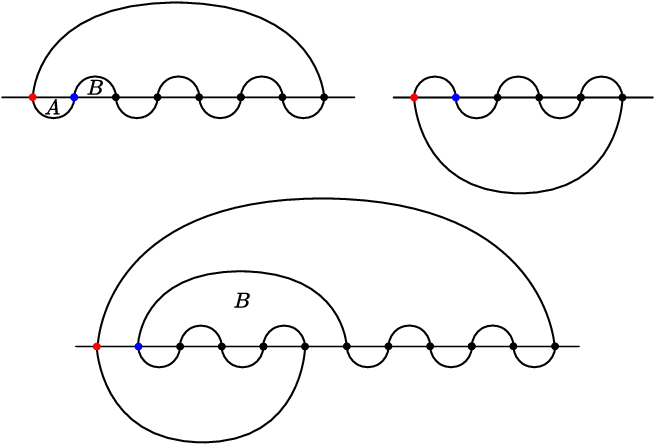}
\end{center}
\caption{Joining a meander in $Q(1^{2},-1^{6})$ and a meander in $Q(2^{2},-1^{8})$ to obtain a meander in $Q(2^{2},1^{2},-1^{10})$.}
\label{f:meanderjoin}
\end{figure}

Given a meander with two adjacent bigons as in the left of Figure~\ref{f:adjacentbigons} and a second meander that has a maximal arc which is immediately adjacent to a bigon as in the right of Figure~\ref{f:adjacentbigons} (note that this is also two adjacent bigons up to cyclic rotation), we can `glue' the second meander into the first joining the red vertices and blue vertices so that the second meander is contained inside bigon $B$ of the first and bigon $A$ of the first is replaced by the maximal arc of the second. If the first meander is in $\calQ(k_{1},\ldots,k_{n},-1^{\kappa+4})$, $\kappa=\sum k_{i}$, and the second is in $\calQ(l_{1},\ldots,l_{m},-1^{\lambda+4})$, $\lambda=\sum l_{i}$, then the resulting meander lies in $\calQ(k_{1},\ldots,k_{n},l_{1},\ldots,l_{m},-1^{\kappa+\lambda+4})$. For example, we can combine a meander in $Q(1^{2},-1^{6})$ and a meander in $Q(2^{2},-1^{8})$ to obtain a meander in $Q(2^{2},1^{2},-1^{10})$, as shown in Figure~\ref{f:meanderjoin}.

Performing this process on minimal meanders in $Q(k^{2},-1^{2k+4})$ (as we did in Figure~\ref{f:meanderjoin}) does not form any zeros of order 0 (i.e., no 4-gons are formed). As such, we see that this process iterates to produce minimal meanders in the strata $Q(k_{1}^{2},\ldots,k_{n}^{2},-1^{\kappa+4})$, $\kappa = \sum_{i=1}^{n}2 k_{i}$. That is, we have the following proposition (this was stated without proof as {\cite[Corollary 6.24]{J0}}).

\begin{proposition}\label{p:symm}
A closed meander in $Q(k_{1}^{2},\ldots,k_{n}^{2},-1^{\kappa+4})$, $\kappa = \sum_{i=1}^{n}2 k_{i}$, requires at least $\kappa + 2n + 2$ crossings.
\end{proposition}

\begin{proof}
As noted above, the gluing procedure does not add any zeros of order zero (4-gons) that are not already contained in the two meanders. So if there are no zeros of order zero in either of the meanders (i.e., they are initially minimal), then the resulting meander is also minimal since it will also contain no zeros of order zero. We observe that we can build a meander in $Q(k_{1}^{2},-1^{2k_{1}+4})$ using $2k_{1}+4$ crossings. To this, we can glue in a meander in $Q(k_{2}^{2},-1^{2k_{2}+4})$ with $2k_{2}+4$ crossings. The resulting meander has $2k_{1}+2k_{2}+4+4-2 = 2k_{1}+2k_{2}+6$ crossings. After this, we still have adjacent bigons and a maximal arc that can be used in the above combination process. We can continue gluing in meanders in $Q(k_{i}^{2},-1^{2k_{i}+4}$ until we reach $i = n$. At each stage, we add $2k_{i}+2$ crossings. Hence, in total we have $2k_{1}+4 + \sum_{i = 2}^{n}(2k_{i}+2) = \kappa +2n +2$ crossings, as claimed.
\end{proof}

A similar argument, using the minimal meanders in $Q(k,-1^{k+4})$ constructed in Proposition~\ref{p:closedmeander1}, allows us to prove the following proposition (this was stated without proof as {\cite[Corollary 6.22]{J0}}).

\begin{proposition}
The minimum number of crossings for a closed meander in the stratum $\calQ(k_{1},\ldots,k_{n},-1^{\kappa+4})$, with $\kappa=\sum k_{i}$, is less than or equal to $4\kappa+2n+2$ crossings.
\end{proposition}

\begin{proof}
The proof is similar to the previous proposition. We repeatedly glue in meanders in $Q(k_{i},-1^{k_{i}+4})$ each of which uses $4k_{i}+4$ crossings. We end with a meander that uses $4k_{1}+4+\sum_{i = 2}^{n}(4k_{i}+2) = 4\kappa+2n+2$ crossings.
\end{proof}

Finally, we observe that in all of the closed meanders we have seen in this work, the number of crossings increases as the difference between the orders of zeros increases. Indeed, for the stratum $Q(k_{1}^{2},\ldots,k_{n}^{2},-1^{\kappa+4})$ the zeros can be paired up with an equal order zero, and we notice that this stratum can be obtained minimally with no additional zeros of order zero. Whereas the stratum $Q(k,-1^{k+4})$ has only one zero which cannot therefore be paired up with a second zero and a minimal meander here required the addition of many zeros of order zero. Hence, we ask the following question:

\begin{question}
For the stratum $\calQ(k_{1},\ldots,k_{n},-1^{\kappa+4})$, with $\kappa=\sum k_{i}$, is the minimal number of crossings for a closed meander explicitly related to the smallest sum of the differences of orders over all possible pairings of the zeros?
\end{question}

If we let this minimal difference be $\sigma$, then the formula $\kappa+2\left\lceil\frac{n}{2}\right\rceil+2+3\sigma$ seems to work for the strata in this paper. Indeed, in $Q(k,-1^{k+4})$ we have $\sigma = k$ and $\kappa = k$ so that $\kappa + 2\left\lceil\frac{n}{2}\right\rceil + 2 + 3\sigma = k + 2 + 2 + 3k = 4k + 4$, the minimum number of crossings. Similarly, for $Q(k_{1},k_{2},-1^{k_{1}+k_{2}+4})$, $k_{1}\geq k_{2}$, we have $\sigma = k_{1}-k_{2}$ and $\kappa = k_{1}+k_{2}$ so that $\kappa + 2\left\lceil\frac{n}{2}\right\rceil + 2 + 3\sigma = k_{1}+k_{2} + 2 + 2 + 3(k_{1}-k_{2}) = 4k_{1}-2k_{2} + 4$ which is again the minimum. Finally, this formula agrees with the result of Proposition~\ref{p:symm} where $\sigma = 0$ by pairing the zeros of order $k_{i}$. How to define $\sigma$ for more complicated strata and what the resulting formula should be is not clear.

\end{document}